\documentclass[a4paper,reqno]{amsart}
\usepackage{palatino, amssymb, amsfonts, latexsym, mathrsfs}
\usepackage{amssymb, amsmath, amscd}
\input xy
\xyoption{all}

\newtheorem{theorem}{Theorem}[section]
\newtheorem{lemma}[theorem]{Lemma}
\newtheorem{prop}[theorem]{Proposition}

\theoremstyle{definition}
\newtheorem{definition}[theorem]{Definition}
\newtheorem{example}[theorem]{Example}
\newtheorem{conjecture}[theorem]{Conjecture}

\theoremstyle{remark}
\newtheorem{remark}[theorem]{Remark}

\newcommand{\Ha}{H_2^{\text{aff}}}

\newcommand{\Hl}{\mathcal K_2^{\mathbf v}}
\newcommand{\Han}{H_n^{\text{aff}}}

\newcommand{\fv}{f_{\mathbf v}}
\newcommand{\Jv}{J_{\mathbf v}}

\numberwithin{equation}{section}

\begin{document}

\title{On the centre of the cyclotomic Hecke algebra of $G(m,1,2)$}

\author{Kevin McGerty}
\address{Department of Mathematics, Imperial College London. }

\date{April, 2010}

\begin{abstract}
We compute the centre of the cyclotomic Hecke algebra attached to $G(m,1,2)$ and show that if $q \neq 1$ it is equal to the image of the centre of the affine Hecke algebra $\Ha$. We also briefly discuss what is known about the relation between the centre of an arbitrary cyclotomic Hecke algebra and the centre of the affine Hecke algebra of type $A$. 
\end{abstract}

\maketitle

\section{On the centre of $\Hl$}
\label{onthecenter}

\subsection{}
We consider quotients of the affine Hecke algebra $\Ha$ of type $A_2$: this is the algebra over $\mathcal A = \mathbb Z[q]$ generated by $T, X_1^{\pm 1},X_2^{\pm 1}$, such that
\begin{enumerate}
\item there is an injective algebra map $\mathcal A[X_1^{\pm 1}, X_2^{\pm 1}] \to \Ha $. 
\item $(T-q)(T+1) = 0$;
\item $TX_1T = qX_2$.
\end{enumerate}
Let $S$ denote the image of the ring of Laurent polynomials $\mathcal A[X_1^{\pm}, X_2^{\pm}]$ in $\Ha$, and let $W$ be the two element group with nontrivial element $s$ which acts on $S$ by interchanging $X_1$ and $X_2$. We write the action as $f \mapsto {^sf}$. For convenience of notation we set $Q = q-1$. Relation ($3$) above is then equivalent to 
\begin{equation}
\label{commutation}
Tf = {^sf}T + Q\frac{f-{^sf}}{1 - X_1X_2^{-1}}, \qquad f \in S.
\end{equation}
It is easy to check from this that the centre of $\Ha$ is $S^W$ the algebra of symmetric functions in the $X_i^{\pm 1}$. 

\subsection{}
Now let $A = \mathbb Z[q^{\pm1},v_1^{\pm1}, v_2^{\pm1}, \ldots, v_m^{\pm1}]$ be a Laurent polynomial ring, and extend the scalars of $\Ha$ to $A\otimes_{\mathcal A} \Ha$. By abuse of notation we will again denote this algebra by $\Ha$, and similarly for the subalgebra $S$. 

\begin{definition}
\label{cycldef}
The cyclotomic Hecke algebra $\Hl$ of type $G(m,1,2)$ is a quotient of $\Ha$: let 
\begin{equation}
\label{charpoly}
\begin{split}
\fv &= (x -v_1)(x-v_2)\ldots(x -v_m) \\
& = \sum_{j=0}^{m} (-1)^{m-j}e_{m-j}x^{j}.
\end{split}
\end{equation}
where the $e_j$ are the elementary symmetric functions in the $v_i$'s. Let $\Jv$ be the two-sided ideal in $\Ha$ generated by $f_1 = \fv(X_1)$ and set $\Hl = \Ha/\Jv$. (Note that our definition coincides with that in \cite{A} up to rescaling, after $v_1,v_2,\ldots,v_m$ have been inverted, except that his ``$q$'' is a square root of ours).
\end{definition}

We say that a polynomial $p$ in $S$ is $m$-\textit{restricted} (or simply \textit{restricted}, when the integer $m$ is understood) if the monomials $X_1^iX_2^j$ occurring with nonzero coefficient in $p$ all satisfy $0 \leq i,j \leq m-1$.  Let $S_m$ denote the space of $m$-restricted polynomials. It is known that the image $R_m$ of $S$ in $\Hl$ is isomorphic as an $A$-module to $S_m$, and moreover $\Hl = R_m \oplus R_mT$ as an $A$-module \textit{i.e.} every element of $\Hl$ can be written uniquely in the form $f +gT$ where $f$ and $g$ are restricted. We refer to this last fact as the ``basis theorem'' for $\Hl$.

\subsection{}
We start with a technical lemma. Let $D$ be the difference operator on $S$ given by
\[
D(f) = (f-{}^sf)/(1-X_1X_2^{-1}), \quad f \in S,
\]
so that the relation \ref{commutation} becomes $Tf = {}^sfT+QD(f)$. Let $D_s$ be the composition $f \mapsto {}^s(-D(f))$, that is,
\[
D_s(f) = (f-{}^sf)/(1-X_1^{-1}X_2).
\]

\begin{lemma}
The operators $D$ and $D_s$ preserve $S_m$ and thus induce $A$-linear maps on $R_m$. Moreover 
\[
D(f) = D_s(f)
\]
if and only if $f ={}^sf$. 
\end{lemma}
\begin{proof}
The proof that $D$ and $D_s$ preserve $R$ is a direct calculation: observe that if $f = X_1^i X_2^j$, then we have 
\[
D(f) = \left\{\begin{array}{cc}X_1^i\sum_{k=0}^{j-i-1} X_1^kX_2^{j-k}, & \text{ if } j>i; \\
-X_1^j\sum_{k=0}^{i-j-1} X_1^kX_2^{i-k}, & \text{ if } j<i.\end{array}\right.
\]
Thus as the highest powers of $X_1$ and $X_2$ occurring in these expressions is $\max\{i,j\}$, it is clear that if $p$ is any $m$-restricted polynomial, so is $D(p)$. Since $^sD$ is the composition of $-D$ with $s$, it clearly also preserves restricted polynomials.

Moreover, note that in $D(f)$ where $f$ is the monomial above, $X_1$ never occurs to the power $\max\{i,j\}$ whereas $X_2$ does, thus for any restricted polynomial $p$, if $D(p)\neq 0$, it has a higher power of $X_2$ occurring than occurs as a power of $X_1$. Thus similarly $D_s(p)$, if nonzero, has a higher power of $X_1$ occurring than occurs as a power of $X_2$. It follows that $D(p) = D_s(p)$ if and only if $D(p) = 0$, and this occurs only if $p = {}^sp$ as claimed. 
\end{proof}

\begin{lemma}
\label{symmetricnecessity}
Let $z \in \mathcal K_2$ and suppose that $z = f+gT$ where $f, g \in R_m$. Then $z$ commutes with $T$ if and only if $f,g \in R_m^W$. 
\end{lemma}
\begin{proof}
The sufficiency of the condition is clear. To see the necessity, we have
\[
\begin{split}
T(f +gT) &= {^sf}T + QD(f) + {^sg}T^2 + QD(g)T \\
& = ({^sf}+ Q{^sg}+QD(g))T +  QD(f)+q{^sg},
\end{split}
\]
(where we write $D$ for the operator on $R_m$ given by the previous lemma). On the other hand, we have
\[
(f+gT)T = (f+Qg)T + qg.
\]
Since $f$ and$g$ are restricted, it follows from the basis theorem for cyclotomic Hecke algebras that we must have:
\[
({^sf}+ Q{^sg}+QD(g)) = f+Qg,
\]
and 
\[
QD(f)+q{^sg} = qg
\]
Thus after rearranging the second of these equations becomes:
\[
QD(f) = q(g-{^sg}).
\] 
Now note that the right-hand side is an eigenvector for the action of $s$ with eigenvalue $-1$, thus so is the left-hand side, whence we get ${}^s(D(f)) = -D(f)$, or equivalently $D(f) = D_s(f)$. By the previous Lemma, this is possible only if $f = {^sf}$, and $D(f) =0$. But then we must also have $g-{^sg}= 0$, and so $f$ and $g$ are symmetric as required.
\end{proof}

\subsection{}
Let $\mathcal Z$ denote the centre of $\mathcal K_2$. From the previous lemma, we see that if $z = f+gT \in \mathcal Z$, then $f, g \in R_m^W$. Since $f\in R_m^W$ is already central, we see that $\mathcal Z = \mathcal Z \cap R_m \oplus \mathcal Z \cap R_mT$, and we are reduced to calculating when $gT$ is central. For this we introduce the following operator:

\begin{definition}
Let $d\colon S \to S$ be the linear map given on monomials by
\[
d(X_1^iX_2^j) =  \left\{\begin{array}{cc}X_1^iX_2^j & \text{ if } i<j; \\-X_1^jX_2^i & \text{ if } i>j  \\0 & \text{ if } i=j.\end{array}\right.
\]
Clearly $d$ preserves $S_m$, and so we may transport it to a map on $R_m$ (which we will also denote by $d$). Clearly the kernel $\text{ker}(d)$ of its action on $R_m$ is $R_m^W$, and since $d^2 = d$, $R_m = R_m^W \oplus d(R_m)$. 
\end{definition}

\begin{lemma}
\label{TpartofZ}
$\mathcal Z \cap R_mT$ is a free $A$-module of rank $m$. 
\end{lemma}
\begin{proof}
Suppose that $gT \in \mathcal Z \cap RT$. We must have $X_1gT = gTX_1$ and $TgT = gT^2$, and these conditions are sufficient. Since $T$ is invertible (indeed $T^{-1} = q^{-1}(T+1-q)$) the second equation is equivalent to $Tg =gT$. By Lemma \ref{symmetricnecessity} this implies that $g \in R^W$, and hence the first equation becomes $(X_1g)T =T( X_1g)$. But then again using Lemma \ref{symmetricnecessity}, we see that $X_1g \in R_m^W$. Let $M$ be the space of such restricted symmetric polynomials:
\[
M = \{g \in R_m^W: X_1g \in R_m^W\}.
\]
We have shown that $\mathcal Z \cap RT = MT$. It is now a linear algebra exercise to check that $M$ is a free $A$-modules of rank $m$. By the paragraph preceding the lemma,  $R_m = R_m^W \oplus d(R_m)$ as an $A$-module. Thus if we let $\phi\colon R_m^W \to d(R_m)$ be the map $g \mapsto d(X_1g)$, we see that $M = \text{ker}(\phi)$. Let $m_{ij}= X_1^iX_2^j + X_1^jX_2^i$ for $i < j$ and $m_{ii} = X_1^iX_2^i$ be the monomial symmetric functions, and let $R^W_{m-1}$ be the span of $\{m_{ij}: 0 \leq i \leq j <m-1\}$. Then we claim that $\phi \colon R^W_{m-1} \to d(R_m)$ is an isomorphism of $A$-modules. Indeed for $j<m-1$ we have:
\[
\phi(m_{ij}) = \left\{\begin{array}{cc}-X_1^iX_2^{j+1}, & \text{ if } j-i \leq 1;\\ -X_1^iX_2^{j+1} +X_1^{i+1}X_2^j, & \text{ if } j-i >1.\end{array}\right.
\]
If we use reverse lexicographical ordering on the bases $\{m_{ij}: 0 \leq i \leq j <m-1\}$ and $\{X_1^iX_2^j: 0 \leq i < j \leq m-1\}$, then the above equations show that the matrix of $\phi_{|R^W_{m-1}}$ with respect to these ordered bases is triangular with diagonal entries equal to $- 1$, thus $\phi_{|R_{m-1}^W}$ is an isomorphism as claimed. 

This immediately implies that $M$ is a free $A$-module of rank $m$. However, we can be more precise and even specify a basis of $M$ by considering the images of $\phi(m_{i(m-1)})$, ($0 \leq i \leq m-1$): set 
\[
p_i = m_{i(m-1)} - (\phi_{R_{m-1}^W})^{-1}(\phi(m_{i(m-1)})),
\]  
then $\{p_0T,p_1T,\ldots p_{m-1}T\}$ is an $A$-basis of $\mathcal Z \cap R_mT$ 

\end{proof}

\begin{example}
\label{3example}
Let $m=3$. In this case we find that the space $\mathcal Z \cap R^WT$ is spanned by $\{p_0T,p_1T,p_2T\}$ where
 \begin{equation}
\label{theZs}
\begin{split}
p_0 = (e_3 -e_1(X_1+X_2) + X_1X_2 + X_1^2 +X_2^2), \\
p_1 = (e_3 -e_1X_1X_2 + X_1X_2(X_1+X_2)),\\
p_2 = (e_3(X_1+X_2) - e_2X_1X_2 +X_1^2X_2^2).
\end{split}
\end{equation}

Thus we have shown that the centre $\mathcal Z$ of $\Hl$ (with $m=3$) is a $9$-dimensional free submodule spanned by $R_3^W$ and $\{p_0T, p_1T, p_2T\}$ (this is exactly as stated in \cite[Section 2]{A}). 
\end{example}

\begin{remark}
It is easy to check directly that Lemma \ref{TpartofZ} implies that the rank of the centre is the number of $m$-multipartitions of $2$. This also follows by passing to the field of fractions of $A$, and using the result of \cite{AK} which shows that in the semisimple case (for any $n$), the centre has rank equal to the number of $m$-multipartition of $n$.  
\end{remark}

\section{On the image of $Z(H_2^\mathrm{aff})$}
\subsection{}
Next we do some simple computations. Let $H_k = \sum_{j=0}^{k}X_1^jX_2^{k-j}$ be the complete symmetric function in $X_1$ and $X_2$ of degree $k$, and let $\mathcal I$ denote the image of the centre of $\Ha$ in $\Hl$. Recall that $f_1= \fv(X_1)$.

\begin{definition}
Let $f_2 = Tf_1T$. Thus $f_2 \in J_v$. 
\end{definition}

\begin{lemma}
\label{explicitcalculation}
a) In $\Ha$, for any $k\geq 2$
\[
TX_1^kT = qX_2^k -Q(X_1X_2)H_{k-2}T
\]
b) We have 
\[
q\fv(X_2) = f_2 + QzT,
\]
where $z = (-1)^{m+1}e_m + (X_1X_2)(\sum_{j=0}^{m-2}(-1)^{j}e_{j}H_{m-2-j}) \in S_m^W$. 
\end{lemma}
\begin{proof}
The proof of a) is a direct calculation using Equation (\ref{commutation}). For b) we have using a) 
\[
qX_2^k = TX_1^kT +Q(X_1X_2)h_{k-2}T, \qquad (k \geq 2).
\]
Moreover, $qX_2 = TX_1T$, and $q = T^2 -QT$, so that
\[
q\fv(X_2) = f_2 + Q\big( (-1)^{m+1}e_m +(X_1X_2)(\sum_{j=2}^{m} (-1)^{m-j}e_{m-j}H_{j-2})\big)T,
\]
as claimed.
\end{proof}

\begin{remark}
Note that one has the well known identity of symmetric functions $\sum_{r=o}^n (-1)^r  e_{n-r}h_r = 0$ for elementary and complete symmetric functions in the \textit{same} set of variables. In the preceding lemma the $e_j$ are symmetric functions in the $v_i$s while the $H_k$s are symmetric in the $X_i$s.
\end{remark}

\begin{prop}
\label{TpartofI}
The elements $QX_1^kz$ lie in $\mathcal I$ for all $k \in \mathbb Z$, and moreover the elements $\{X_1^{k-1}z: 0 \leq k \leq m-1\}$ are linearly independent.
\end{prop}
\begin{proof}
From the previous lemma we have
\[
\begin{split}
QX_1^kzT &= qX_1^k\fv(X_2) - X_1^kf_2 \\
& = qX_1^k\fv(X_2) + qX_2^k\fv(X_1) -(qX_2^kf_1 +X_1^kf_2) \\
& \in qX_1^k\fv(X_2) + qX_2^k\fv(X_1) + \Jv.
\end{split}
\]
hence we see that $QX_1^kzT \in \Hl = \Ha/\Jv$ is in the image of the centre of $\Ha$ as required.

It remains to show that the elements $\{X_1^{k-1} z: 0 \leq k \leq  m-1\}$ are linearly independent. Since $A$ is an integral domain, we see that $X_1^{k-1}zT \in \mathcal Z$, and hence $X_1^{k-1}z \in R_m^W$ by Lemma \ref{symmetricnecessity}. Now we have
\[
X_1^{k-1}z = (-1)^{m+1}e_mX_1^{k-1} - (X_1^{k}X_2)(\sum_{j=0}^{m-2}(-1)^{j+1}e_{j}H_{m-2-j}).
\]
Now consider this expression for $X_1^{k-1}z$ as a linear combination of the monomial symmetric functions $m_{ij}$ lying in $R_m$: by considering the terms involving $X_2^{m-1}$ we see that the coefficient of $m_{j(m-1)}$ is zero unless $j=k$ in which case it is $1$. It follows immediately that the $X_1^{k-1} z$ in the range $0\leq k \leq m-1$ are linearly independent.
\end{proof}

\subsection{}
We can now combine the above results to establish our main theorem.

\begin{theorem}
Let $B$ denote the localization $A$ where $Q=q-1$ is inverted. Then, over $B$, the centre of $\Ha$ surjects onto the centre of $\Hl$.
\end{theorem}  
\begin{proof}
It is clearly sufficient to show that $p_iT$ lies in $\mathcal I$, where $p_i \in R_m^W$, $(0 \leq i \leq m-1)$, is as in Lemma \ref{TpartofZ}. Since we have inverted $Q$, Proposition \ref{TpartofI} shows that we have $X_1^kzT \in \mathcal I$ for all $k \in \mathbb Z$. Now by the proof of the previous proposition, we also know that the coefficient of $m_{j(m-1)}$ in $X_1^{k-1}z$ ($0 \leq k \leq m-1$) in the basis $\{m_{ij}\}$ of restricted monomial symmetric functions is $\delta_{jk}$, and the same is true, by definition, for the $p_k$. Since $\mathcal I \subset \mathcal Z$,  we can write $X_1^kzT$ as a linear combination of the elements $p_iT$, hence it follows immediately that $p_kT = X_1^{k-1}zT$, and we are done.
\end{proof}

\begin{example}
We consider again the case $m=3$, keeping the notation of the previous example. Then $z=p_1$, and it is easy to check that $X_1p_1 = p_2$, and similarly

\[
\begin{split}
X_1^{-1}p_1 &= e_3X_1^{-1} - e_1X_2 + X_2(X_1+X_2) \\
&= (e_2-e_1X_1+ X_1^2) - e_1X_2 + X_2(X_1+X_2) \\
&= e_2 -e_1(X_1+X_2)+(X_1^2+X_1X_2+X_2^2) \\
&=p_0
\end{split}
\]
so that $X_1^{k-1}z = p_{k}$ for $0 \leq k \leq 2$.

\end{example}

\section{Comments on the general case}
\subsection{}
We wish to consider the following conjecture.

\begin{conjecture}
\label{Centreconjecture}
Let $H_n^{\text{aff}}$ be the affine Hecke algebra with coefficients extended to $B$, the ring \[
A = \mathbb Z[q^{\pm1},v_1^{\pm 1}, \ldots, v_m^{\pm 1}]\]
with $Q = q-1$ inverted. Let $\psi_m\colon H_n^{\text{aff}} \to \mathcal K_n$ be the quotient map. Then
\[
\psi_m(Z(H_n^{\text{aff}})) = Z(\mathcal K_m).
\]
\end{conjecture}
\noindent
In fact, it may be easier (and as useful) to show this in the case where $H_n^{\text{aff}}$ is defined over a field $F$, and the parameter $q$ is not equal to $1$. 

\begin{remark}
As pointed out in \cite{A}, if we specialise to $q=1$, \textit{i.e.} $Q=0$, then the image of the centre of the affine Hecke algebra does \textit{not} necessarily surjects onto the centre of the specialised Ariki-Koike algebra (for $\Hl$, if we say require $X_1$ to satisfy $X_1^3 -1 =0$, then at $q=1$ this is just the group algebra of the complex reflection group $G(3,1,2)$ which has a $9$ dimensional centre -- the specialisation of the centre of $\Hl$ -- whereas the images of $X_1+X_2$, $X_1X_2$ only generate a $6$-dimensional subalgebra). Of course at $q=1$ one should instead consider the degenerate algebra.

While the above conjecture is certainly not new, it does not seem to be explicitly stated in the literature, and some of existing literature is unclear as to its status: the counterexample of \cite{A} is quoted in \cite{M} in a fashion which makes it appear it is more general than \cite{A} intended to imply\footnote{The author thank Profs Ariki and Mathas for helping him sort out this confusion.}.

\end{remark}

We list the following evidence for the conjecture:

\begin{itemize}
\item \cite{AK} shows that the conjecture holds in the semisimple case (they also show explicitly the conditions on the parameters under which the Ariki-Koike algebras is semisimple).
\item This note establishes the case $n=2$.
\item In an orthogonal direction, Francis and Graham \cite{FG} have verified the conjecture for the case of the finite Hecke algebra of type $A$, \textit{i.e.} the case $m=1$. 
\end{itemize}

\subsection{}

We end with another result which supports the conjecture when we work over $F$, an algebraically closed field of characteristic zero. Let $\mathfrak H_n$ be the degenerate, or graded, affine Hecke algebra of type $A$. As a vector space it is isomorphic to $\mathbb C[S_n]\otimes \mathbb C[r][x_1,\ldots,x_n]$. If $\mathscr T$ denotes the algebraic torus with regular functions $\mathcal O = \mathbb C[q^{\pm1}, X_1^{\pm 1},\ldots, X_n^{\pm 1}]$, and $\mathfrak t\oplus \mathbb C$ the vector space with functions $\bar{\mathcal O} = \mathbb C[r][x_1,x_2,\ldots,x_n]$ note that we may use the group structure on $\mathscr T$ to identify $\mathfrak t \oplus \mathbb C$ with the tangent space of $\mathscr T$ at any point. 

 Now let $I$ be a maximal ideal of the centre $\mathcal Z$ of $\Han$. Thus $I$ corresponds to an $S_n$-orbit $\Sigma$ in $\mathscr T$. We want to consider the completions $\widehat{\mathcal Z}$ and $\widehat{\Han}$ with respect to $I$. Assume that all of the coordinates of the elements of $\Sigma$ are equal to a power of $q$. In this case, by choosing a logarithm of $q$ we may attach to $\Sigma$ a maximal ideal $\mathfrak I$ of the centre $\mathfrak Z = \mathbb C[r][x_1,\ldots,x_n]^{S_n}$ of $\mathfrak H_n$, and consider the corresponding completions $\widehat{\mathfrak H}_n$ and $\widehat{\mathfrak Z}$. The algebras $\widehat{\mathcal Z}$ and $\widehat{\mathfrak Z}$ are then naturally isomorphic and Lusztig \cite[9.3]{L} has shown that the $I$-adic completion $\widehat{\Han}$ of $\Han$ is isomorphic to the corresponding completion $\widehat{\mathfrak H}_n$ of $\mathfrak H_n$ as algebras over $\widehat{\mathcal Z} \cong \widehat{\mathfrak Z}$. Moreover, the  isomorphism restricts to give an isomorphism between the (completed) commutative subalgebras $\widehat{\mathcal O}$ and $\widehat{\bar{\mathcal O}}$.

Let $\mathcal K_n$ be the cyclotomic Hecke algebra over $F$, where let us moreover assume that $v_i = q^{a_i}$ for some integers $a_i$, ($1 \leq i \leq m$), and $q \in F$ has infinite order. It is known \cite{DM} that representation theory of a general cyclotomic quotient can be reduced to the case where the $v_i$ are of this form\footnote{The restriction that $q$ have infinite order however, is genuinely restrictive.}. Recently, Brundan \cite{Br} has established the analogue of Conjecture \ref{Centreconjecture} for the degenerate cyclotomic Hecke algebras $\mathfrak K_n$, where $\mathfrak K_n$ is the quotient of $\mathfrak H_n$ by the two-sided ideal generated by $f_r(x_1)$ where $f_r(t) = \prod_{i=1}^m (t-a_i)$.

Now we may decompose a cyclotomic Hecke algebra according to the spectrum of the image of $\mathcal Z$, which is certainly a central subalgebra (in fact, it follows from the work of Lyle-Mathas \cite{LM} that this decomposition yields the blocks of the cyclotomic algebra, but we do not need that here). For each of these subalgebras, the above discussion shows that the summands of the cyclotomic and degenerate cyclotomic Hecke algebras are isomorphic as they are quotients of $\widehat{\Han}$ and $\widehat{\mathfrak H}_n$ which are annihilated by a finite power of some maximal idea $I$ (resp. $\mathfrak I$) of the centre $\mathcal Z$ (resp. $\mathfrak Z$), since Lusztig's isomorphisms are maps of $\widehat{\mathcal Z}\cong \widehat{\mathfrak Z}$-algebras. Thus we can deduce from Brundan's work the following result:

\begin{prop}
Let $F$ be an algebraically closed field of characteristic zero, $q \in F$ of infinite order and $\mathcal K_n$ a cyclotomic Hecke algebra with $v_i=q^{a_i}$ for some integers $a_i$. Then the centre of $\mathcal K_n$ is equal to the image of the centre of $\Han$.
\end{prop}

\begin{remark}
It is shown in \cite{A1} that the cyclotomic Hecke algebra is semisimple precisely when the polynomial:
\[
P(q,\mathbf v) = \prod_{1 \leq i< j \leq m}\big( \prod_{-n < a < n}(q^av_i-v_j)\big)\cdot \prod_{k=1}^n(1+q+\ldots + q^{k-1}),
\]
is nonvanishing. Thus the cyclotomic Hecke algebra need not be semisimple even when $q$ is not a root of unity, so this result includes cases which are not covered by the results of \cite{AK}. It should also be noted that Brundan and Kleshchev \cite{BK} have recently shown that if $F$ is any field of characteristic zero and $q$ is not a root of unity then the cyclotomic and degenerate cyclotomic Hecke algebras are isomorphic. One can presumably use their results to extend the above Proposition to this more general situation.
\end{remark}

\noindent
\textit{Acknowledgements}: We thank Prof. Ariki and Prof. Mathas for helpful correspondence. We also thank Anthony Henderson for suggesting that the world would not be significantly worse off if this note were to be posted online.

\end{document}